\documentclass[a4paper,12pt]{amsart}

\usepackage[T1]{fontenc}
\usepackage{amsmath,amssymb,amsthm,mathtools,hyperref,geometry,cleveref,mathrsfs,xcolor,microtype,tabularx,array,multirow,float,orcidlink}
\usepackage{enumitem}

\allowdisplaybreaks

\definecolor{DarkBlue}{rgb}{0.1,0.1,0.55}
\definecolor{DarkRed}{rgb}{0.55,0.1,0.1}

\hypersetup{
  colorlinks=true,
  linkcolor=DarkBlue,
  citecolor=DarkRed,
  urlcolor=DarkBlue,
  pdftitle={Rigidity Theorems for Self-Shrinkers of the Mean Curvature Flow},
  pdfauthor={Fagui Li and Yuhang Zhao}
}

\makeatletter
\renewcommand{\@defaultbiblabelstyle}[1]{[#1]}
\makeatother

\numberwithin{equation}{section}

\newtheorem{theorem}{Theorem}[section]

\newtheorem{lemma}[theorem]{Lemma}
\newtheorem{corollary}[theorem]{Corollary}

\newtheorem*{spherechern}{Chern's conjecture for minimal hypersurfaces}
\newtheorem*{selfshrinkerchern}{Chern's conjecture for self-shrinkers}
\newtheorem*{cwyconjecture}{Cheng--Wei--Yano conjecture}
\newtheorem*{cwyproblem}{Cheng--Wei--Yano problem}
\newtheorem*{yauconjecture}{Yau's conjecture for self-shrinkers}
\newtheorem{problem}{Problem}[section]
\theoremstyle{definition}

\theoremstyle{remark}
\newtheorem{remark}{Remark}[section]

\newcommand{\R}{\mathbb{R}}
\newcommand{\Sn}{\mathbb{S}}
\newcommand{\Sphere}{\mathbb{S}}
\newcommand{\cL}{\mathcal L}

\newcommand{\inner}[2]{\left\langle #1,#2\right\rangle}

\title[Rigidity Theorems for Self-Shrinkers]{Rigidity Theorems for Self-Shrinkers of the Mean Curvature Flow}

\author[F. G. Li]{Fagui Li}
\address{Frontier Interdisciplinary Domain, Beijing Institute of Technology, Zhuhai, Guangdong 519088, P. R. China}
\email{lifagui@bitzh.edu.cn}

\author[Y. Zhao]{Yuhang Zhao${}^{*}$}
\address{School of Mathematics, Nanjing University, Nanjing 210093, P. R. China}
\email{yuhangzhao@smail.nju.edu.cn}

\thanks{F. G. Li is partially supported by NSFC (No. 12271040 and 12501061),  the Guangdong Provincial Association for Science and Technology Youth Talent Support Program (No. SKXRC2026413) and the Research Start-up Funding of Beijing Institute of Technology (No. 5640011253301)}

\thanks{$^{*}$ Corresponding author.}
\begin{document}

\begin{abstract}
In this paper, we prove a spectral upper-pinching theorem for complete properly immersed
self-shrinking hypersurfaces. Our argument is inspired by the second author's recent work\cite{Zhao2025}.  If \(\lambda_\rho(\Sigma)\geq\lambda>0\) and
\(S=|A|^2<1+\lambda\), then \(\Sigma\) is either a hyperplane or a generalized
round cylinder.  In the properly embedded
case, the Ding--Xin and Brendle--Tsiamis weighted Poincar\'e estimate gives
\(\lambda_\rho(\Sigma)\geq1/2\).  Consequently, the pointwise upper pinching
\(S<3/2\) forces \(\Sigma\) to be a hyperplane or a generalized round
cylinder.  For embedded self-shrinking surfaces in \(\mathbb R^3\), we also
obtain the endpoint case \(S\leq3/2\).  These results remove the lower
pointwise pinching assumption in the corresponding embedded upper-pinching
range and improve the ranges in earlier work of Ding--Xin~\cite{DingXin2014}, Cheng--Wei~\cite{ChengWei2015}, and
Lei--Xu--Xu~\cite{LeiXuXu2020}. 
\end{abstract}

\maketitle

\section{Introduction}

Let \(X:\Sigma^n\to\R^{n+1}\) be an immersed two-sided hypersurface, and fix a
global unit normal \(N\).  Unless explicitly stated otherwise, all hypersurfaces in
this paper are connected and have no boundary.  We write \(d\mu\) for the induced
Riemannian measure on \(\Sigma\), and we use \(\inner{\cdot}{\cdot}\) and
\(|\cdot|\) for the Euclidean inner product and the induced norms.  The symbol
\(\overline\nabla\) denotes the Euclidean connection, while \(\nabla\),
\(\operatorname{div}_\Sigma\), and
\[
  \Delta=\operatorname{div}_\Sigma\nabla
\]
denote the intrinsic Levi-Civita connection, divergence, and Laplacian on
\(\Sigma\).  With this sign convention, \(-\Delta\) is nonnegative on compact
manifolds.  Commas denote covariant derivatives with respect to \(\nabla\), and
\(Y^T\) denotes the tangential projection of an ambient vector \(Y\) onto
\(T\Sigma\).

For a local orthonormal frame \(\{e_i\}_{i=1}^n\), we use
\[
  h_{ij}=\inner{\overline\nabla_{e_i}e_j}{N},
  \qquad H=\sum_{i=1}^n h_{ii}.
\]
Let \(A:T\Sigma\to T\Sigma\) be the corresponding shape operator,
\[
  A(e_i)=\sum_j h_{ij}e_j,
\]
and denote its eigenvalues by \(\kappa_1,\ldots,\kappa_n\) when needed.  We put
\[
  S=|A|^2=\sum_{i,j}h_{ij}^2,
  \qquad
  \rho=e^{-|X|^2/2},
\]
and define the drift Laplacian by
\[
  \cL u=\Delta u-\inner{X}{\nabla u}.
\]
The hypersurface \(\Sigma\) is called a self-shrinker if
\begin{equation}\label{eq:shrinker}
  H+\inner{X}{N}=0.
\end{equation}
In this normalization, the flat model is the hyperplane through the origin, for
which \(S\equiv0\).  The nonflat standard models are the generalized round
cylinders, namely hypersurfaces which, after an ambient orthogonal transformation
of \(\R^{n+1}\), have the form
\[
  \Sphere^k(\sqrt{k})\times\R^{n-k}\subset\R^{n+1},
  \qquad 1\leq k\leq n.
\]
Each nonflat standard model satisfies \(S\equiv1\).  When the flat case is also
allowed, we state the conclusion as ``a hyperplane through the origin or a
generalized round cylinder.''

For a properly embedded hypersurface \(\Sigma\subset\R^{n+1}\), we use the
following convention of Brendle--Tsiamis~\cite[Section~2]{BrendleTsiamis2024}:
\(\R^{n+1}\setminus\Sigma\) has two connected components, denoted by \(\Omega\)
and \(\widetilde\Omega\), and \(N\) is the outward-pointing unit normal to
\(\Omega\).  The signed mean curvature \(H\) is always computed with this choice
of normal.  Thus no separate two-sided assumption is needed in the properly
embedded statements below.  In the immersed case, two-sided still means that the
normal line bundle is trivial and a global unit normal has been fixed.

Self-shrinkers model Type~I singularities of the mean curvature flow, beginning
with the work of Huisken~\cite{Huisken1990,Huisken1993}, and they are central in
the entropy and generic-singularity theory of Colding--Minicozzi~\cite{ColdingMinicozzi2012Generic}.
A basic rigidity question asks when a pointwise upper bound for
\(S=|A|^2\) forces a self-shrinker to be one of the standard models.  The guiding
analogy is Chern's gap problem for minimal hypersurfaces in spheres.

We recall only the spherical model needed for motivation.  If
\(M^n\subset\Sphere^{n+1}(1)\) is a closed minimal hypersurface and
\(S_M=|A_M|^2\), then the scalar curvature of \(M\) is \(R_M=n(n-1)-S_M\).  Thus
Chern's constant-scalar-curvature problem can be phrased in terms of constant
\(S_M\).
\begin{spherechern} 
	Assume $M^n\subset \Sphere^{n+1}(1)$ is a compact minimal hypersurface.
	\begin{enumerate}[label=\textup{(\roman*)}]
		\item If $S_M$ is constant and $n\le S_M\le 2n$, then either $S_M= n$ or $S_M= 2n$.
		\item If $n\le S_M\le 2n$, then either $S_M\equiv n$ or $S_M\equiv 2n$.
	\end{enumerate}
\end{spherechern}
The first gap was proved by Simons~\cite{Simons1968} and
Chern--do Carmo--Kobayashi~\cite{ChernDoCarmoKobayashi1970}: if \(S_M\leq n\),
then \(M\) is either a great sphere or a Clifford hypersurface.  Beyond this
range, the second-gap problem was initiated by Peng--Terng~\cite{PengTerng1983MathAnn}.
We use the spherical theory only as motivation; some representative progress includes
Yang--Cheng~\cite{YangCheng1998}, Ding--Xin~\cite{DingXin2011}, Lei--Xu--Xu~\cite{LeiXuXu2017JFA},
Cheng--Wei--Yamashiro~\cite{ChengWeiYamashiro2021},
  Tang--Wei--Yan~\cite{TangWeiYan2020},
Tang--Yan~\cite{TangYan2023}, Ge--Tan--Yan--Zhang~\cite{GeTanYanZhang2025},
Ge--Liu--Luo--Yan~\cite{GeLiuLuoYan2026}, and He--Xu--Zhao~\cite{HeXuZhao2026}, etc.
The compact embedded version suggested by this analogy is recorded in
Problem~\ref{prob:spherical-embedded-pinching}.

Motivated by this picture, we consider the following Chern-type upper gap for
self-shrinkers.
\begin{selfshrinkerchern}
For each \(n\geq2\), there exists a constant \(c(n)>0\) with the following
property.  Let \(X:\Sigma^n\to\R^{n+1}\) be a complete properly immersed self-shrinking hypersurface satisfying \eqref{eq:shrinker}.  If
\[
  S=|A|^2<1+c(n)
  \quad\text{on }\Sigma,
\]
then \(\Sigma\) is one of the following: a hyperplane through the origin, a
generalized round cylinder \(\Sphere^k(\sqrt{k})\times\R^{n-k}\),
\(1\leq k\leq n\), or a product \(\Gamma\times\R^{n-1}\), where
\(\Gamma\subset\R^2\) is a non-round Abresch--Langer self-shrinking curve.  In
particular, if \(\Sigma\) is properly embedded, then only the hyperplane and the
generalized round cylinders can occur.
\end{selfshrinkerchern}

The first gap for self-shrinkers is classical.  Le--\v{S}e\v{s}um~\cite{LeSesum2011}
proved an early gap theorem, and Cao--Li~\cite{CaoLi2013} established the sharp
first gap: a complete self-shrinker with polynomial volume growth and \(S\leq1\)
is a hyperplane or a generalized round cylinder.  Ding--Xin~\cite{DingXin2014}
obtained Simons-type identities and several pointwise and integral rigidity
results, and Ding~\cite{Ding2018} later strengthened an integral rigidity theorem
for the second fundamental form.  Cheng--Peng~\cite{ChengPeng2015} used a
generalized maximum principle for the drift operator to remove polynomial volume
growth from related rigidity statements.  Lei--Xu--Xu~\cite{LeiXuXu2020} proved
cylindrical rigidity under the pinching condition \(0\leq S-1\leq1/18\).  In the
properly immersed noncompact category, a lower pinching condition is unavoidable
for a cylinder-only conclusion: products of non-round Abresch--Langer~\cite{AbreschLanger1986}
shrinking curves with Euclidean factors give noncylindrical examples with \(S\)
arbitrarily close to \(1\); see also Mantegazza~\cite{Mantegazza2011}.  Recently,
Zhao~\cite[Theorem~1.1]{Zhao2025} proved that the lower pointwise restriction can
be dropped in the compact arbitrary-codimension case: if
\(X:\Sigma^n\to\R^{n+p}\), \(n\geq2\), is a closed self-shrinker and
\(|\mathrm{II}|^2\leq 1+1/(10\pi(n+2))\), then \(\Sigma\) is the round sphere
\(\Sphere^n(\sqrt n)\).  Related higher-codimension rigidity has also been
obtained for Lagrangian self-shrinkers by Li--Wang~\cite{LiWang2017}, and for
self-shrinkers of \(r\)-mean curvature flows by Alencar--Bessa--Silva Neto~\cite{AlencarBessaSilvaNeto2025}.

The constant-\(S\) case is the closest self-shrinker analogue of Chern's original
constant-scalar-curvature problem.  Cheng--Wei--Yano~\cite{ChengWeiYano2023}
formulate the following classification conjecture.
\begin{cwyconjecture}[\cite{ChengWeiYano2023}]
Let \(X:M^n\to\R^{n+1}\) be a complete self-shrinker.  If the squared norm \(S\)
of the second fundamental form is constant, then \(X(M)\) is isometric to one of
the following standard models:
\begin{enumerate}[label=\textnormal{(\arabic*)},leftmargin=2em,itemsep=2pt,topsep=2pt]
\item the hyperplane \(\R^n\);
\item a round cylinder \(\Sphere^k(\sqrt{k})\times\R^{n-k}\), \(1\leq k\leq n-1\);
\item the round sphere \(\Sphere^n(\sqrt n)\).
\end{enumerate}
Equivalently, constant \(S\) should force \(S\equiv0\) or \(S\equiv1\).
\end{cwyconjecture}
Since the full constant-\(S\) classification is difficult in general dimension,
Cheng--Wei--Yano isolate the following second-gap problem.
\begin{cwyproblem}[\cite{ChengWeiYano2023}]
Let \(X:M^n\to\R^{n+1}\) be a complete self-shrinker.  If \(S\) is constant,
then either \(X(M)\) is one of the three standard models listed in the preceding
conjecture, or
$
  S\geq2.
$
Equivalently, there should be no nonstandard complete self-shrinker with constant
\(S\) in the open range \(1<S<2\).
\end{cwyproblem}
This problem is precisely the constant-\(S\) form of the Chern-type second gap for
self-shrinkers.  It is known in dimension two: Ding--Xin~\cite{DingXin2014}
proved it under properness, and Cheng--Ogata~\cite{ChengOgata2016} removed the
polynomial volume-growth assumption.  Cheng--Li--Wei~\cite{ChengLiWei2022}
proved the three-dimensional constant-\(S\) classification under the additional
assumption that \(f_4=\sum_i\kappa_i^4\) is constant.  Quantitative partial
results include the gap \(S<10/7\) of Cheng--Wei~\cite{ChengWei2015} and the
range \(S<1.83379\) of Cheng--Wei--Yano~\cite{ChengWeiYano2023} when both \(S\)
and \(f_3=\sum_i\kappa_i^3\) are constant.

The main analytic input in this paper is a weighted Poincar\'e inequality for the
drift Laplacian.  For a complete properly immersed self-shrinker, define
\[
  \lambda_\rho(\Sigma)=
  \inf_{\substack{f\not\equiv0,\; \int_\Sigma f\rho\,d\mu=0\\
  \int_\Sigma(f^2+|\nabla f|^2)\rho\,d\mu<\infty}}
  \frac{\int_\Sigma |\nabla f|^2\rho\,d\mu}
       {\int_\Sigma f^2\rho\,d\mu}.
\]
The Euclidean coordinate functions \(X_\alpha=\langle X,e_\alpha\rangle\) satisfy
\(\cL X_\alpha=-X_\alpha\).  Weighted integration by parts gives
\(\int_\Sigma X_\alpha\rho\,d\mu=0\), and at least one coordinate function is
nonconstant; hence \(\lambda_\rho(\Sigma)\leq1\).  Ding--Xin~\cite{DingXin2013}
proved \(\lambda_\rho\geq1/2\) in the compact embedded case, and
Brendle--Tsiamis~\cite{BrendleTsiamis2024} proved the same estimate for properly
embedded, possibly noncompact self-shrinkers.  In the normalization
\eqref{eq:shrinker}, their estimate reads
\[
  \int_\Sigma |\nabla f|^2\rho\,d\mu
  \geq \frac12\int_\Sigma f^2\rho\,d\mu
\]
for every zero-Gaussian-mean function with finite weighted \(H^1\) norm.  This is
the self-shrinker analogue of the Choi--Wang~\cite{ChoiWang1983}
first-eigenvalue estimate for embedded minimal hypersurfaces in spheres.  In our
earlier work~\cite{LiZhao2026}, this estimate was applied to normal coordinate
functions to obtain Gaussian \(L^2\)-curvature lower bounds.  Here the decisive
test function is instead built from the positive and negative parts of the mean
curvature.

The following theorem is the spectral core of the paper.  It proves the
Chern-type gap with the explicit pinching constant \(c=\lambda\) whenever a
weighted Poincar\'e lower bound \(\lambda_\rho(\Sigma)\geq\lambda\) is available.
\begin{theorem}\label{thm:spectral-pinching}
Let \(X:\Sigma^n\to\R^{n+1}\) be a complete properly immersed self-shrinker
satisfying \eqref{eq:shrinker}.  Suppose that \(\lambda_\rho(\Sigma)\geq\lambda>0\)
and
\[
  S<1+\lambda
  \quad\text{on }\Sigma .
\]
Then 
\(\Sigma\) is one of the following:
\begin{enumerate}[label=\textnormal{(\alph*)},leftmargin=2em,itemsep=2pt,topsep=2pt]
\item a hyperplane through the origin;
\item a generalized round cylinder
\[
  \Sphere^k(\sqrt{k})\times\R^{n-k}
\]
for some \(1\leq k\leq n\).
\end{enumerate}
\end{theorem}

Thus Theorem~\ref{thm:spectral-pinching} gives the following
unconditional upper-pinching result.
\begin{theorem}\label{thm:embedded-main}
Let \(\Sigma^n\subset\R^{n+1}\) be a complete properly embedded self-shrinker
satisfying \eqref{eq:shrinker}.  If
\[
  S<\frac32
  \quad\text{on }\Sigma,
\]
then \(\Sigma\) is either a hyperplane through the origin or a generalized round
cylinder
$
  \Sphere^k(\sqrt{k})\times\R^{n-k}
$
for some \(1\leq k\leq n\).
\end{theorem}

\begin{remark}\label{rem:comparison}
Theorem~\ref{thm:embedded-main} removes the lower pointwise pinching condition
from the complete embedded hypersurface setting in the range \(S<3/2\).  It
should be compared with the theorem of Lei--Xu--Xu~\cite{LeiXuXu2020}, which
assumes \(0\leq S-1\leq1/18\), and with the constant-\(S\) gap \(S<10/7\) of
Cheng--Wei~\cite{ChengWei2015}.  The price paid here is embeddedness: it is used
first to obtain the weighted Poincar\'e estimate and then to exclude the
Abresch--Langer alternatives from the mean-convex classification.
\end{remark}

In dimension two the endpoint can also be included.
\begin{theorem}\label{thm:surface-borderline}
Let \(\Sigma^2\subset\R^3\) be a complete properly embedded self-shrinker
satisfying \eqref{eq:shrinker}.  If
\[
  S\leq\frac32
  \quad\text{on }\Sigma,
\]
then \(\Sigma\) is either a hyperplane through the origin, the round cylinder
\(\Sphere^1(1)\times\R\), or the round sphere \(\Sphere^2(\sqrt2)\).
\end{theorem}
In particular, if Yau's conjecture for self-shrinkers holds, then the spectral theorem of this
paper gives the expected sharp second-gap range.
\begin{corollary}\label{cor:yau-conditional}
Assume that the weighted spectral equality
$
  \lambda_\rho(\Sigma)=1
$
holds for every complete properly immersed self-shrinker in the class under consideration. If $S<2$, then 
\(\Sigma\) is either a hyperplane through the origin or a generalized round
cylinder
$
  \Sphere^k(\sqrt{k})\times\R^{n-k}
$
for some \(1\leq k\leq n\). In particular, when $S$ is constant, the Cheng--Wei--Yano
second-gap problem has an affirmative answer in this class.
\end{corollary}
The paper is organized as follows.  Section~\ref{sec:preliminaries} records the
weighted Poincar\'e estimate, the volume-growth facts, the mean-convex
classification, the two-dimensional constant-\(S\) classification, and the drift
identities used later.  Section~\ref{sec:proof} proves the fixed-sign lemma and
then proves Theorems~\ref{thm:spectral-pinching}, \ref{thm:embedded-main}, and
\ref{thm:surface-borderline}.  Section~\ref{sec:concluding} explains the role of
embeddedness, records the sharp spectral conjecture and its consequences, and
lists related open problems.

\section{Preliminaries}\label{sec:preliminaries}

We collect several standard facts for self-shrinkers in the normalization used here.
We first recall the weighted Poincar\'e inequality.

\begin{theorem}[Ding--Xin~\cite{DingXin2013}; Brendle--Tsiamis~\cite{BrendleTsiamis2024}]\label{thm:eigenvalue}
Let \(\Sigma^n\subset\R^{n+1}\) be a complete properly embedded self-shrinker satisfying \eqref{eq:shrinker}.  If \(f\in H^1_{\mathrm{loc}}(\Sigma)\) satisfies
\[
  \int_\Sigma(f^2+|\nabla f|^2)\rho\,d\mu<\infty,
  \qquad
  \int_\Sigma f\rho\,d\mu=0,
\]
then
\[
  \int_\Sigma|\nabla f|^2\rho\,d\mu
  \geq
  \frac12\int_\Sigma f^2\rho\,d\mu.
\]
\end{theorem}

The compact embedded case is due to Ding--Xin~\cite[Theorem~1.3]{DingXin2013}; the properly embedded noncompact case is due to Brendle--Tsiamis~\cite[Theorem~1.1]{BrendleTsiamis2024}.  No separate polynomial volume-growth assumption is needed under properness.  Ding--Xin~\cite{DingXin2013} showed that complete noncompact properly immersed self-shrinkers have Euclidean volume growth, and Cheng--Zhou~\cite{ChengZhou2013} proved the equivalence between properness of the immersion, polynomial volume growth, and finiteness of the Gaussian weighted volume for complete immersed self-shrinkers in Euclidean space.  Brendle--Tsiamis use the normalization \(H=\frac12\langle x,\nu\rangle\), the weight \(e^{-|x|^2/4}\), and the constant \(1/4\).  Rescaling by \(X=x/\sqrt2\), and reversing the normal if necessary, gives \eqref{eq:shrinker}; the weight becomes \(e^{-|X|^2/2}\), while the Dirichlet energy scales so that the Poincar\'e constant becomes \(1/2\).

\begin{theorem}[Abresch--Langer~\cite{AbreschLanger1986}; Huisken~\cite{Huisken1990,Huisken1993}; Colding--Minicozzi~\cite{ColdingMinicozzi2012Generic}]\label{thm:huisken-mean-convex}
Let \(\Sigma^n\subset\R^{n+1}\) be a complete properly immersed self-shrinker satisfying \eqref{eq:shrinker} and  having polynomial volume growth and with bounded second fundamental form.  If the scalar mean curvature has a fixed sign, then \(\Sigma\) is one of the following:
\begin{enumerate}[label=\textnormal{(\alph*)},leftmargin=2em,itemsep=2pt,topsep=2pt]
\item a hyperplane through the origin;
\item a generalized round cylinder
\[
  \Sphere^k(\sqrt{k})\times\R^{n-k}
\]
for some \(1\leq k\leq n\);
\item a product
\[
  \Gamma\times\R^{n-1},
\]
where \(\Gamma\subset\R^2\) is a non-round Abresch--Langer self-shrinking curve.
\end{enumerate}
If \(\Sigma\) is embedded, then the last alternative is excluded.
\end{theorem}
For \(n=1\), this is the Abresch--Langer classification, with the embedded case reducing to the line and the round circle.  For \(n\geq2\), the compact and the bounded-curvature complete cases are due to Huisken, and Colding--Minicozzi removed the bounded-curvature assumption in the embedded noncompact case.  In the immersed category, the additional products over Abresch--Langer curves are precisely the alternatives that embeddedness rules out.

\begin{lemma}\label{lem:finite-H1}
Let \(X:\Sigma^n\to\R^{n+1}\) be a complete properly immersed self-shrinker satisfying \eqref{eq:shrinker}.  If \(S\) is bounded, then
\[
  0<\int_\Sigma\rho\,d\mu<\infty,
  \qquad
  \int_\Sigma(H^2+|\nabla H|^2)\rho\,d\mu<\infty.
\]
\end{lemma}

\begin{proof}
The shrinker equation gives \(|H|=|\langle X,N\rangle|\leq|X|\).  Differentiating \(H=-\langle X,N\rangle\) and using the Weingarten equation \(\overline\nabla_{e_i}N=-\sum_j h_{ij}e_j\), we obtain
\[
  \nabla_i H=\sum_j h_{ij}\langle X,e_j\rangle.
\]
Consequently,
\[
  |\nabla H|\leq |A|\,|X^T|\leq\sqrt S\,|X|.
\]
If \(S\) is bounded, then \(H^2+|\nabla H|^2\leq C(1+|X|^2)\).

A complete properly immersed self-shrinker has Euclidean volume growth by Ding--Xin's theorem~\cite{DingXin2013}: there is a constant \(C\) such that
\[
  \mu(\Sigma\cap B_R)\leq C(1+R)^n,
  \qquad R\geq1.
\]
Decomposing \(\Sigma\) into \(\Sigma\cap B_1\) and the annuli \(\Sigma\cap(B_{j+1}\setminus B_j)\) for \(j\geq1\), we find, for every \(q\geq0\),
\[
  \int_\Sigma(1+|X|)^q e^{-|X|^2/2}\,d\mu
  \leq C_q\left(1+\sum_{j=1}^\infty (1+j)^{n+q}e^{-j^2/2}\right)<\infty.
\]
Taking \(q=0\) and \(q=2\) proves the assertions.  The Gaussian volume is positive because \(\Sigma\) is nonempty and \(\rho>0\).
\end{proof}

\begin{lemma}\label{lem:basic-identities}
Let \(\Sigma^n\subset\R^{n+1}\) be a self-shrinker satisfying \eqref{eq:shrinker}.  Then
\begin{equation}\label{eq:L-H}
  \cL H=(1-S)H.
\end{equation}
Moreover, for compactly supported test functions one has the weighted Green formula
\begin{equation}\label{eq:weighted-green}
  \int_\Sigma u\cL v\,\rho\,d\mu
  =-\int_\Sigma\langle\nabla u,\nabla v\rangle\rho\,d\mu.
\end{equation}
\end{lemma}

\begin{proof}
We first derive \eqref{eq:L-H}.  Fix a point and choose a local orthonormal frame that is geodesic at that point.  From the computation in the proof of Lemma~\ref{lem:finite-H1},
\[
  H_{,i}=\sum_k h_{ik}\langle X,e_k\rangle.
\]
Differentiating once more and using
\[
  e_j\langle X,e_k\rangle
  =\delta_{jk}+h_{jk}\langle X,N\rangle
  =\delta_{jk}-Hh_{jk},
\]
we obtain
\[
  H_{,ij}
  =\sum_k h_{ik,j}\langle X,e_k\rangle+h_{ij}
   -H\sum_k h_{ik}h_{jk}.
\]
Taking the trace and using the Codazzi equations gives
\[
  \Delta H=\langle X,\nabla H\rangle+H-HS,
\]
which is \eqref{eq:L-H}.  Since
\[
  \operatorname{div}_\Sigma(\rho\nabla v)=\rho\cL v,
\]
\eqref{eq:weighted-green} follows from the ordinary divergence theorem.
\end{proof}

The following two-dimensional constant-\(S\) classification, due to Ding--Xin~\cite{DingXin2014} and Cheng--Ogata~\cite{ChengOgata2016}, will also be used.

\begin{theorem}[Ding--Xin~\cite{DingXin2014}; Cheng--Ogata~\cite{ChengOgata2016}]\label{thm:surface-constant-S}
Let \(\Sigma^2\subset\R^3\) be a complete properly immersed self-shrinker satisfying \eqref{eq:shrinker}.  If \(S\) is constant, then \(\Sigma\) is a hyperplane, the round cylinder \(\Sn^1(1)\times\R\), or the round sphere \(\Sn^2(\sqrt2)\).  In particular, 
a complete properly immersed self-shrinking surface in \(\R^3\) with constant \(S\) has \(S\in\{0,1\}\).
\end{theorem}

We also need the next lemma, which is the point that removes the last alternative in
Theorem~\ref{thm:huisken-mean-convex}.

\begin{lemma}\label{lem:AL-spectral-obstruction}
Let \(\Gamma\subset\R^2\) be a non-round closed Abresch--Langer
self-shrinking curve, possibly traversed more than once, and let
\[
  \Sigma=\Gamma\times\R^d,
  \qquad d\geq0.
\]
Let \(p\) be the rotation number of \(\Gamma\), and let \(q\) be the number of
periods of its curvature along one traversal.  Then
\[
  \frac12<\frac pq<\frac1{\sqrt2},
  \qquad p\geq2,
  \qquad q\geq3.
\]
If \(k\) is the positive curvature of \(\Gamma\), and
\[
  a=\min_\Gamma k,
  \qquad
  b=\max_\Gamma k,
\]
then
\begin{equation}\label{eq:AL-gap-chain}
  \lambda_\rho(\Sigma)
  =\lambda_\rho(\Gamma)
  <\frac1{p^2}
  <\frac{2}{p^2+1}
  <b^2-1.
\end{equation}
Moreover, \(S_\Sigma=k^2\).  Consequently,
\begin{equation}\label{eq:AL-obstruction}
  \max_\Sigma S_\Sigma>1+\lambda_\rho(\Sigma).
\end{equation}
\end{lemma}

\begin{proof}
Choose the orientation and the unit normal \(N\) so that the Frenet equations
are
\[
  T_s=kN,
  \qquad
  N_s=-kT,
  \qquad
  k>0,
\]
and the self-shrinker equation is
\[
  k+\langle\Gamma,N\rangle=0.
\]
Here \(s\) is arclength.  Let \(\theta\) be the turning-angle parameter, so that
\(d\theta/ds=k\).  Differentiating the self-shrinker equation gives
\[
  k_s=k\langle\Gamma,T\rangle,
  \qquad
  k_\theta=\langle\Gamma,T\rangle.
\]
Differentiating once more with respect to \(\theta\), we obtain
\begin{equation}\label{eq:AL-curvature-ode}
  k_{\theta\theta}+k=\frac1k.
\end{equation}
Multiplication by \(2k_\theta\) shows that
\begin{equation}\label{eq:AL-first-integral}
  k_\theta^2+k^2-2\log k=E
\end{equation}
for a constant \(E\).  Since the function \(t^2-2\log t\) has its unique minimum
\(1\) at \(t=1\), a nonconstant periodic solution has \(E>1\).  Its two turning
values therefore satisfy
\[
  0<a<1<b,
  \qquad
  a^2-2\log a=b^2-2\log b=E.
\]
Moreover, \eqref{eq:AL-first-integral} shows that these are the only turning
values; uniqueness for \eqref{eq:AL-curvature-ode} then implies that \(k\) is
strictly monotone between consecutive extrema.
The total change of \(\theta\) along one traversal is \(2\pi p\), and the
curvature has period
\begin{equation}\label{eq:AL-period}
  T_0=\frac{2\pi p}{q}.
\end{equation}
These are the standard integer parameters in the Abresch--Langer
classification~\cite{AbreschLanger1986}.  The displayed inequalities and
integrality imply \(p\geq2\) and \(q\geq3\): indeed, \(p=1\) would force the
integer \(q\) to lie strictly between \(\sqrt2\) and \(2\).  If the curve is a
multiple traversal, \(p\) and \(q\) are multiplied by the same integer; the
argument below uses the parameters of that actual traversal.

We next reduce the weighted eigenvalue problem to one dimension.  From the two
components of \(\Gamma\) in the Frenet frame and
\eqref{eq:AL-first-integral},
\[
  |\Gamma|^2
  =\langle\Gamma,T\rangle^2+\langle\Gamma,N\rangle^2
  =k_\theta^2+k^2
  =E+2\log k.
\]
Consequently,
\begin{equation}\label{eq:AL-weight-measure}
  \rho\,ds=e^{-E/2}\frac{d\theta}{k^2}.
\end{equation}
For a smooth function \(u=u(\theta)\),
\[
  u_s=ku_\theta,
  \qquad
  u_{ss}=kk_\theta u_\theta+k^2u_{\theta\theta},
  \qquad
  \langle\Gamma,T\rangle u_s=kk_\theta u_\theta.
\]
Thus the drift Laplacian on \(\Gamma\) is
\begin{equation}\label{eq:AL-drift-reduction}
  \mathcal L_\Gamma u=k^2u_{\theta\theta}.
\end{equation}
It follows from \eqref{eq:AL-weight-measure} that
\begin{equation}\label{eq:AL-rayleigh}
  \lambda_\rho(\Gamma)
  =\inf_{\substack{0\not\equiv u\in H^1_{\rm per}(0,2\pi p)\\
      \int_0^{2\pi p}u k^{-2}\,d\theta=0}}
  \frac{\displaystyle\int_0^{2\pi p}u_\theta^2\,d\theta}
       {\displaystyle\int_0^{2\pi p}u^2k^{-2}\,d\theta}.
\end{equation}
Equivalently, its eigenfunctions solve the periodic Sturm--Liouville problem
\[
  -u_{\theta\theta}=\lambda k^{-2}u,
  \qquad
  u(\theta+2\pi p)=u(\theta),
  \qquad
  u_\theta(\theta+2\pi p)=u_\theta(\theta).
\]

Set \(u(\theta)=\sin(\theta/p)\).  Since \(k^{-2}\) has period \(T_0\),
\[
\begin{aligned}
  \int_0^{2\pi p}u k^{-2}\,d\theta
  &=\int_0^{T_0}k(\tau)^{-2}
    \sum_{j=0}^{q-1}
    \sin\left(\frac{\tau}{p}+\frac{2\pi j}{q}\right)d\tau
  =0.
\end{aligned}
\]
Moreover, because \(q\geq3\),
\[
  \sum_{j=0}^{q-1}
  \sin^2\left(\frac{\tau}{p}+\frac{2\pi j}{q}\right)=\frac q2.
\]
Dividing \eqref{eq:AL-curvature-ode} by \(k\), integrating over one
curvature period, and using periodicity gives
\begin{equation}\label{eq:AL-weight-period}
  \int_0^{T_0}\frac{d\theta}{k^2}
  =T_0+\int_0^{T_0}\left(\frac{k_\theta}{k}\right)^2d\theta
  >T_0,
\end{equation}
where the inequality is strict because \(\Gamma\) is not round.  Therefore the
Rayleigh quotient of \(u\) satisfies
\[
\begin{aligned}
  \lambda_\rho(\Gamma)
  &\leq
  \frac{\displaystyle\int_0^{2\pi p}p^{-2}\cos^2(\theta/p)\,d\theta}
       {\displaystyle\int_0^{2\pi p}\sin^2(\theta/p)k^{-2}\,d\theta}  \\
  &=\frac{\pi/p}
    {\displaystyle(q/2)\int_0^{T_0}k^{-2}\,d\theta}
  <\frac{\pi/p}{(q/2)T_0}
  =\frac1{p^2}.
\end{aligned}
\]

We now prove the lower bound for \(b\).  Put \(v=k-1\).  Equation
\eqref{eq:AL-curvature-ode} becomes
\begin{equation}\label{eq:AL-v-equation}
  v_{\theta\theta}+\left(1+\frac1k\right)v=0.
\end{equation}
After translating \(\theta\), on one curvature period \(k\) increases strictly
from \(a\) to \(b\) and then decreases strictly back to \(a\).  Since
\(a<1<b\), the function \(v\) has exactly one zero on each monotone branch.
Both zeros are simple: at a point where \(k=1\),
\eqref{eq:AL-first-integral} gives \(k_\theta^2=E-1>0\).  Thus \(v\) has exactly
two simple zeros on that period.  Let the two nodal intervals have lengths \(L_1\) and \(L_2\), so
that \(L_1+L_2=T_0\).  Multiplying \eqref{eq:AL-v-equation} by \(v\) on a
nodal interval and using the Dirichlet Poincar\'e inequality gives
\[
  \frac{\pi^2}{L_i^2}\int v^2\,d\theta
  \leq\int v_\theta^2\,d\theta
  =\int\left(1+\frac1k\right)v^2\,d\theta
  \leq\left(1+\frac1a\right)\int v^2\,d\theta.
\]
Thus
\[
  L_i\geq\frac{\pi}{\sqrt{1+a^{-1}}},
  \qquad i=1,2.
\]
Using \eqref{eq:AL-period}, we obtain
\[
  \frac{2\pi p}{q}=T_0=L_1+L_2
  \geq\frac{2\pi}{\sqrt{1+a^{-1}}},
\]
and hence
\begin{equation}\label{eq:AL-min-curvature-bound}
  a\leq\frac{p^2}{q^2-p^2}.
\end{equation}
Since \(q^2>2p^2\) and both sides are integers,
\[
  D:=q^2-2p^2\geq1.
\]
It follows from \eqref{eq:AL-min-curvature-bound} that
\begin{equation}\label{eq:AL-a-gap}
  1-a\geq
  1-\frac{p^2}{q^2-p^2}
  =\frac{D}{p^2+D}
  \geq\frac1{p^2+1}.
\end{equation}

It remains to compare the two oscillation amplitudes.  Let
\[
  \Phi(t)=t^2-2\log t.
\]
Then \(\Phi(a)=\Phi(b)\), and \(\Phi\) is strictly increasing on
\((1,\infty)\).  If \(s=1-a\in(0,1)\), then
\[
\begin{aligned}
  \Phi(1+s)-\Phi(1-s)
  &=4s-2\log\frac{1+s}{1-s}<0,
\end{aligned}
\]
because
\[
  \log\frac{1+s}{1-s}
  =2\int_0^s\frac{dt}{1-t^2}>2s.
\]
If \(b-1\leq1-a=s\), then \(b\leq1+s\), and therefore
\[
  \Phi(b)\leq\Phi(1+s)<\Phi(1-s)=\Phi(a),
\]
a contradiction.  Hence \(b-1>1-a\).  Combining this with
\eqref{eq:AL-a-gap} yields
\[
  b^2-1=(b-1)(b+1)
  >2(b-1)
  >2(1-a)
  \geq\frac{2}{p^2+1}.
\]
Since \(p\geq2\),
\[
  \frac{2}{p^2+1}>\frac1{p^2}.
\]
This proves the three strict inequalities in \eqref{eq:AL-gap-chain} for the
curve itself.

\end{proof}

\section{Proof of the pinching theorems}\label{sec:proof}

In this section we prove the main pinching theorems.  The first step is an abstract fixed-sign lemma for the mean curvature under a weighted spectral gap and an upper curvature pinching.  We then prove the properly immersed spectral theorem, derive the embedded consequences, and treat the two-dimensional endpoint.

\begin{lemma}\label{lem:H-fixed-sign}
Let \(X:\Sigma^n\to\R^{n+1}\) be a complete properly immersed self-shrinker satisfying \eqref{eq:shrinker}.  Suppose that, for some \(\lambda>0\),
\[
  \lambda_\rho(\Sigma)\geq\lambda
\]
and
\[
  S<1+\lambda
  \quad\text{on }\Sigma .
\]
Then \(H\) does not change sign.
\end{lemma}

\begin{proof}
The curvature assumption gives a uniform bound for \(S\).  By Lemma~\ref{lem:finite-H1},
\[
  H\in H^1(\rho\,d\mu),
  \qquad
  \int_\Sigma\rho\,d\mu<\infty.
\]
In particular, \(H\in L^1(\rho\,d\mu)\).  Suppose that \(H\) changes sign, and set
\[
  H_+=\max\{H,0\},
  \qquad
  H_-=\max\{-H,0\}.
\]
The truncation maps are Lipschitz, so \(H_+,H_-\in H^1(\rho\,d\mu)\).  Moreover
\[
  0<a:=\int_\Sigma H_+\rho\,d\mu<\infty,
  \qquad
  0<b:=\int_\Sigma H_-\rho\,d\mu<\infty.
\]
Define
\[
  \varphi=H_+-\frac{a}{b}H_-.
\]
Then \(\varphi\in H^1(\rho\,d\mu)\) and \(\int_\Sigma\varphi\rho\,d\mu=0\).  Hence the spectral-gap assumption gives
\begin{equation}\label{eq:phi-poincare}
  \lambda\int_\Sigma\varphi^2\rho\,d\mu
  \leq
  \int_\Sigma|\nabla\varphi|^2\rho\,d\mu .
\end{equation}

We next justify the positive- and negative-part testing identities.  This is the only point where both the noncompactness of \(\Sigma\) and the nonsmoothness of the truncations have to be treated carefully.  For \(0<\varepsilon<1\), set
\[
  P_\varepsilon(t)=
  \begin{cases}
    \sqrt{t^2+\varepsilon^2}-\varepsilon, & t>0,\\[2mm]
    0, & t\leq0 .
  \end{cases}
\]
Then \(P_\varepsilon\in C^{1,1}(\mathbb R)\), \(0\leq P_\varepsilon(t)\leq t_+\), \(P_\varepsilon(t)\to t_+\), and
\[
  P_\varepsilon'(t)=
  \begin{cases}
    \dfrac{t}{\sqrt{t^2+\varepsilon^2}}, & t>0,\\[2mm]
    0, & t\leq0 .
  \end{cases}
\]
In particular, \(0\leq P_\varepsilon'\leq1\) and \(P_\varepsilon'(t)\to\chi_{\{t>0\}}\) for every \(t\).  The function is only \(C^{1,1}\), but this is sufficient for the chain rule.  Equivalently, one may first replace \(P_\varepsilon\) by smooth convex mollifications with the same uniform Lipschitz bounds and then pass to the limit.

Let \(\eta_R\) be a standard ambient cutoff: \(\eta_R\equiv1\) on \(B_R\), \(\operatorname{spt}\eta_R\subset B_{2R}\), \(|\nabla\eta_R|\leq C/R\), and \(\eta_R\) has compact support on \(\Sigma\) by properness.  Since \(P_\varepsilon(H)\) is compactly supported after multiplication by \(\eta_R^2\), the compactly supported weighted Green formula may be applied with
\[
  u=\eta_R^2P_\varepsilon(H),
  \qquad
  v=H.
\]
Using \(\nabla P_\varepsilon(H)=P_\varepsilon'(H)\nabla H\), we obtain
\begin{align}
  \int_\Sigma \eta_R^2P_\varepsilon(H)\cL H\,\rho\,d\mu
  &=-\int_\Sigma \eta_R^2P_\varepsilon'(H)|\nabla H|^2\rho\,d\mu  \notag\\
  &\quad -2\int_\Sigma \eta_RP_\varepsilon(H)\langle\nabla\eta_R,\nabla H\rangle\rho\,d\mu .
  \label{eq:test-positive-eps-R}
\end{align}
For fixed \(\varepsilon\), the cutoff error tends to zero as \(R\to\infty\).  Indeed, if \(A_R=\Sigma\cap(B_{2R}\setminus B_R)\), then
\[
\begin{aligned}
&\left|2\int_\Sigma \eta_RP_\varepsilon(H)
  \langle\nabla\eta_R,\nabla H\rangle\rho\,d\mu\right| \\
&\qquad\leq
\frac{C}{R}
\left(\int_{A_R} P_\varepsilon(H)^2\rho\,d\mu\right)^{1/2}
\left(\int_{A_R} |\nabla H|^2\rho\,d\mu\right)^{1/2} \\
&\qquad\leq
\frac{C}{R}
\left(\int_{A_R} H^2\rho\,d\mu\right)^{1/2}
\left(\int_{A_R} |\nabla H|^2\rho\,d\mu\right)^{1/2},
\end{aligned}
\]
which tends to zero because \(H,\nabla H\in L^2(\rho\,d\mu)\).  Moreover,
\[
  |P_\varepsilon(H)\cL H|
  = |P_\varepsilon(H)(1-S)H|
  \leq C H^2,
\]
where \(C\) depends only on the bound for \(S\).  Hence dominated convergence gives, after letting \(R\to\infty\) in \eqref{eq:test-positive-eps-R},
\[
  \int_\Sigma P_\varepsilon(H)\cL H\,\rho\,d\mu
  =-
  \int_\Sigma P_\varepsilon'(H)|\nabla H|^2\rho\,d\mu .
\]
Letting \(\varepsilon\to0\), the left-hand side converges to \(\int_\Sigma H_+\cL H\,\rho\,d\mu\).  On the right-hand side,
\[
  P_\varepsilon'(H)|\nabla H|^2
  \longrightarrow
  \chi_{\{H>0\}}|\nabla H|^2
\]
pointwise and is bounded by \(|\nabla H|^2\).  By the Sobolev chain rule for truncations,
\[
  \nabla H_+=\chi_{\{H>0\}}\nabla H
  \quad\text{a.e.},
\]
or, equivalently, by the Stampacchia property \(\nabla H=0\) a.e. on \(\{H=0\}\).  Thus dominated convergence yields
\[
  \int_\Sigma H_+\cL H\,\rho\,d\mu
  =-
  \int_\Sigma |\nabla H_+|^2\rho\,d\mu .
\]
Using \(\cL H=(1-S)H\) and \(HH_+=H_+^2\), we obtain
\begin{equation}\label{eq:H-plus-energy}
  \int_\Sigma |\nabla H_+|^2\rho\,d\mu
  =
  \int_\Sigma (S-1)H_+^2\rho\,d\mu .
\end{equation}

For the negative part, we apply the same argument to \(P_\varepsilon(-H)\).  In the weighted Green formula we take
\[
  u=\eta_R^2P_\varepsilon(-H),
  \qquad
  v=H.
\]
Since \(\nabla P_\varepsilon(-H)=-P_\varepsilon'(-H)\nabla H\), the main term has the opposite sign:
\begin{align*}
  \int_\Sigma \eta_R^2P_\varepsilon(-H)\cL H\,\rho\,d\mu
  &=\int_\Sigma \eta_R^2P_\varepsilon'(-H)|\nabla H|^2\rho\,d\mu  \\
  &\quad -2\int_\Sigma \eta_RP_\varepsilon(-H)
  \langle\nabla\eta_R,\nabla H\rangle\rho\,d\mu .
\end{align*}
The cutoff term is treated exactly as above, using \(0\leq P_\varepsilon(-H)\leq H_-\leq |H|\).  Letting \(R\to\infty\) and then \(\varepsilon\to0\), and using
\[
  \nabla H_-=-\chi_{\{H<0\}}\nabla H
  \quad\text{a.e.},
\]
we obtain
\[
  \int_\Sigma H_-\cL H\,\rho\,d\mu
  =
  \int_\Sigma |\nabla H_-|^2\rho\,d\mu .
\]
Since \(HH_-=-H_-^2\), \eqref{eq:L-H} gives
\begin{equation}\label{eq:H-minus-energy}
  \int_\Sigma |\nabla H_-|^2\rho\,d\mu
  =
  \int_\Sigma (S-1)H_-^2\rho\,d\mu .
\end{equation}

Because \(H_+H_-=0\) and \(\nabla H_+\cdot\nabla H_-=0\) almost everywhere,
\[
  \varphi^2=H_+^2+\frac{a^2}{b^2}H_-^2,
  \qquad
  |\nabla\varphi|^2=|\nabla H_+|^2+\frac{a^2}{b^2}|\nabla H_-|^2.
\]
Combining this with \eqref{eq:H-plus-energy} and \eqref{eq:H-minus-energy}, we find
\[
  \lambda\int_\Sigma\varphi^2\rho\,d\mu
  -
  \int_\Sigma|\nabla\varphi|^2\rho\,d\mu
  =
  \int_\Sigma(1+\lambda-S)
  \left(H_+^2+\frac{a^2}{b^2}H_-^2\right)\rho\,d\mu.
\]
The right-hand side is strictly positive because \(S<1+\lambda\), \(\rho>0\), and \(H_+\) and \(H_-\) are both nontrivial.  This contradicts \eqref{eq:phi-poincare}.  Hence \(H\) cannot change sign.
\end{proof}

\begin{proof}[\textbf{Proof of Theorem~\ref{thm:spectral-pinching}}]
        The fixed-sign conclusion follows directly from Lemma~\ref{lem:H-fixed-sign}.  Since \(S<1+\lambda\), the second fundamental form is bounded.  Moreover, complete properly immersed self-shrinkers have polynomial volume growth by the volume-growth theorem quoted after Theorem~\ref{thm:eigenvalue}.  Hence, using the mean-convex classification in Theorem~\ref{thm:huisken-mean-convex} and Lemma~\ref{lem:AL-spectral-obstruction}, we can conclude that $\Sigma $ is either a hyperplane through the origin or a generalized round cylinder. This completes the proof of the theorem.
\end{proof}

\begin{proof}[\textbf{Proof of Theorem~\ref{thm:embedded-main}}]
Since \(\Sigma\) is properly embedded, Theorem~\ref{thm:eigenvalue} gives \(\lambda_\rho(\Sigma)\geq1/2\). Then the desired conclusion  follows from  Theorem~\ref{thm:spectral-pinching}.
\end{proof}

\begin{proof}[\textbf{Proof of Theorem~\ref{thm:surface-borderline}}]
Theorem~\ref{thm:eigenvalue} applies because \(\Sigma\) is properly embedded.  We first show that \(H\) cannot change sign.  Suppose the contrary, and define \(H_+\), \(H_-\), \(a\), \(b\), and \(\varphi\) as in the proof of Lemma~\ref{lem:H-fixed-sign}.  The same testing identities give
\[
  \frac12\int_\Sigma\varphi^2\rho\,d\mu
  \leq
  \int_\Sigma|\nabla\varphi|^2\rho\,d\mu
  \leq
  \frac12\int_\Sigma\varphi^2\rho\,d\mu.
\]
Therefore equality holds, and
\[
  0=
  \int_\Sigma
  \left(\frac32-S\right)
  \left(H_+^2+\frac{a^2}{b^2}H_-^2\right)\rho\,d\mu.
\]
The integrand is continuous and nonnegative, so it vanishes everywhere.  Hence \(S=3/2\) on \(\{H\neq0\}\).  Since \(H\) solves the linear elliptic equation \(\cL H=(1-S)H\), we have $\Delta \left(\exp^{-\frac{|X|^2}{4}}H\right)=\left( 1-S+ \exp^{\frac{|X|^2}{4}}  \Delta  \exp^{-\frac{|X|^2}{4}}  \right)  \exp^{-\frac{|X|^2}{4}}     H$. Hence, the nodal set theory and the fact that $H$ is not identically zero imply  that \(\{H\neq0\}\) is dense.  By continuity, \(S\equiv3/2\) on \(\Sigma\).  This contradicts Theorem~\ref{thm:surface-constant-S}, which says that a complete properly immersed self-shrinking surface in \(\R^3\) with constant \(S\) must have \(S\in\{0,1\}\).  Hence \(H\) has a fixed sign.  Applying Theorem~\ref{thm:huisken-mean-convex} and using embeddedness  to exclude the Abresch--Langer product alternative gives the stated classification.  In dimension two the remaining alternatives are precisely the hyperplane through the origin, \(\Sphere^1(1)\times\R\), and \(\Sphere^2(\sqrt2)\).
\end{proof}

\begin{proof}[\textbf{Proof of Corollary~\ref{cor:yau-conditional}}]
If \(S<2\) and \(\lambda_\rho(\Sigma)=1\), Theorem~\ref{thm:spectral-pinching}
applies with \(\lambda=1\).  Hence the only properly immersed alternatives are
the hyperplane and  the generalized round cylinders.
\end{proof}
\section{Concluding remarks and related problems}\label{sec:concluding}

We close by clarifying where embeddedness enters the proof.
Theorem~\ref{thm:spectral-pinching} itself is an immersed result: under a
weighted spectral gap and an upper curvature pinching assumption, the signed mean
curvature has a fixed sign.  This part of the argument therefore does not require
embeddedness.  Embeddedness is used only in the subsequent passage to the
properly embedded classification.  First, it provides the lower bound
\(\lambda_\rho\geq1/2\) through the Ding--Xin and Brendle--Tsiamis estimate.
Second, once mean-convexity is obtained, it rules out the non-round
Abresch--Langer products
\(\Gamma\times\R^{n-1}\), which may occur in the immersed mean-convex
classification but are not properly embedded.

The sharp form of the spectral input would be the following Gaussian analogue of
Yau's first-eigenvalue conjecture for closed embedded minimal hypersurfaces in
spheres.
\begin{yauconjecture}
Let \(\Sigma^n\subset\R^{n+1}\) be a complete properly embedded self-shrinker
satisfying \eqref{eq:shrinker}.  Then
$
  \lambda_\rho(\Sigma)=1.
$
\end{yauconjecture}
As explained in the introduction, the coordinate functions give
\(\lambda_\rho(\Sigma)\leq1\).  Thus the conjecture asserts that properly
embedded self-shrinkers have the same sharp Gaussian Poincar\'e constant as the
standard models.  The value \(1\) is attained by the hyperplane, by the round
sphere \(\Sphere^n(\sqrt n)\), and by every generalized round cylinder.  In the
usual normalization of self-shrinkers, Matinpour~\cite{Matinpour2025} recently
proved that the first weighted eigenvalue is \(1/2\) for self-shrinkers in
\(\R^3\) invariant under the dihedral group \(\mathbb D_{g+1}\) or the prismatic
group \(\mathbb D_{g+1}\times\mathbb Z_2\); after rescaling to the normalization
\eqref{eq:shrinker}, this corresponds to \(\lambda_\rho=1\).

This conjecture is intended as the self-shrinker counterpart of Yau's
first-eigenvalue conjecture for embedded minimal hypersurfaces in spheres; see
Yau~\cite{Yau1982} and Schoen--Yau~\cite{SchoenYau1994}.  Supporting results in
the spherical setting include work of Choe--Soret~\cite{ChoeSoret2009} on several
symmetric minimal surfaces in \(\Sphere^3\), the results of
Tang--Yan~\cite{TangYan2013} and Tang--Xie--Yan~\cite{TangXieYan2014} for
isoparametric hypersurfaces and their focal submanifolds, and quantitative
improvements over the Choi--Wang lower bound~\cite{ChoiWang1983} due to
Duncan--Sire--Spruck~\cite{DuncanSireSpruck2024}, Jim\'enez--Tapia
Chinchay--Zhou~\cite{JimenezChinchayZhou2026}, and Zhao~\cite{Zhao2026FirstEigenvalue}.
For background on minimal surfaces in \(\Sphere^3\), see Brendle~\cite{Brendle2013Survey}.

The preceding discussion also suggests several related problems.
\begin{problem}\label{prob:remove-embeddedness}

Let \(X:\Sigma^n\to\R^{n+1}\), \(n\geq2\), be a complete properly immersed
self-shrinker, and let \(|\mathrm{II}|^2\) denote the
squared norm of its vector-valued second fundamental form.  Does there exist a
constant \(c(n)>0\) such that
\[
  |\mathrm{II}|^2<1+c(n)
\]
forces \(X(\Sigma)\)  to be a hyperplane through
the origin, a generalized round cylinder
\(\Sphere^k(\sqrt{k})\times\R^{n-k}\), \(1\leq k\leq n\), or a product
\(\Gamma\times\R^{n-1}\), where \(\Gamma\subset\R^2\) is a non-round
Abresch--Langer self-shrinking curve?

\end{problem}

\begin{problem}\label{prob:high-codim-complete}
Let \(X:\Sigma^n\to\R^{n+p}\), \(n\geq2\), be a complete properly immersed
self-shrinker of arbitrary codimension, and let \(|\mathrm{II}|^2\) denote the
squared norm of its vector-valued second fundamental form.  Does there exist a
constant \(c(n)>0\) such that
\[
  |\mathrm{II}|^2<1+c(n)
\]
forces \(X(\Sigma)\), up to an ambient orthogonal transformation, to be a hyperplane through
the origin, a generalized round cylinder
\(\Sphere^k(\sqrt{k})\times\R^{n-k}\), \(1\leq k\leq n\), or a product
\(\Gamma\times\R^{n-1}\), where \(\Gamma\subset\R^2\) is a non-round
Abresch--Langer self-shrinking curve?
\end{problem}
The compact counterpart in arbitrary codimension is known by
Zhao~\cite[Theorem~1.1]{Zhao2025}: if \(\Sigma^n\) is closed and
\(|\mathrm{II}|^2\leq 1+1/(10\pi(n+2))\), then \(\Sigma\) is the round sphere
\(\Sphere^n(\sqrt n)\).  The problem above asks whether an analogous one-sided
upper gap persists in the complete noncompact setting, where cylinders and
Abresch--Langer products have to be allowed.

\begin{problem}\label{prob:spherical-embedded-pinching}
Let \(M^n\subset\Sphere^{n+1}(1)\) be a closed embedded minimal hypersurface, and
let \(S_M=|A_M|^2\).  Does there exist a constant \(c(n)>0\) such that the
pointwise upper pinching
\[
  S_M<n+c(n)
\]
forces \(M\) to be one of the first-gap models, namely either the great sphere
with \(S_M\equiv0\), or a Clifford hypersurface
\[
  \Sphere^k\left(\sqrt{\frac{k}{n}}\right)\times
  \Sphere^{n-k}\left(\sqrt{\frac{n-k}{n}}\right)
  \subset\Sphere^{n+1}(1),
  \qquad 1\leq k\leq n-1,
\]
with \(S_M\equiv n\)?
\end{problem}
This is the compact spherical analogue of the upper-pinching questions for
complete self-shrinkers considered above.  The classical
Simons--Chern--do Carmo--Kobayashi theorem gives the conclusion under the
stronger assumption \(S_M\leq n\).  The problem asks whether, in the embedded
category, a uniform one-sided gap persists slightly above \(n\) without assuming
\(S_M\) to be constant.


\begin{thebibliography}{99}

\bibitem{AbreschLanger1986}
U.~Abresch and J.~Langer,
\emph{The normalized curve shortening flow and homothetic solutions},
J. Differential Geom. \textbf{23} (1986), no.~2, 175--196.

\bibitem{AlencarBessaSilvaNeto2025}
H.~Alencar, G.~P.~Bessa, and G.~Silva Neto,
\emph{Gap theorems for complete self-shrinkers of {$r$}-mean curvature flows},
J. Funct. Anal. \textbf{289} (2025), Paper No.~110920.

\bibitem{Brendle2013Survey}
S.~Brendle,
\emph{Minimal surfaces in {$S^3$}: a survey of recent results},
Bull. Math. Sci. \textbf{3} (2013), 133--171.

\bibitem{BrendleTsiamis2024}
S.~Brendle and R.~Tsiamis,
\emph{Eigenvalue estimates on shrinkers},
preprint, arXiv:2402.11803, 2024.

\bibitem{CaoLi2013}
H.-D.~Cao and H.~Z.~Li,
\emph{A gap theorem for self-shrinkers of the mean curvature flow in arbitrary codimension},
Calc. Var. Partial Differential Equations \textbf{46} (2013), no.~3--4, 879--889.

\bibitem{ChernDoCarmoKobayashi1970}
S.-S.~Chern, M.~do Carmo, and S.~Kobayashi,
\emph{Minimal submanifolds of a sphere with second fundamental form of constant length},
in \emph{Functional Analysis and Related Fields}, Springer, Berlin, 1970,
pp.~59--75.

\bibitem{ChengLiWei2022}
Q.-M.~Cheng, Z.~Li, and G.~X.~Wei,
\emph{Complete self-shrinkers with constant norm of the second fundamental form},
Math. Z. \textbf{300} (2022), no.~1, 995--1018.

\bibitem{ChengOgata2016}
Q.-M.~Cheng and S.~Ogata,
\emph{2-dimensional complete self-shrinkers in {$\mathbb R^3$}},
Math. Z. \textbf{284} (2016), no.~1--2, 537--542.

\bibitem{ChengPeng2015}
Q.-M.~Cheng and Y.~J.~Peng,
\emph{Complete self-shrinkers of the mean curvature flow},
Calc. Var. Partial Differential Equations \textbf{52} (2015), no.~3--4, 497--506.

\bibitem{ChengWei2015}
Q.-M.~Cheng and G.~X.~Wei,
\emph{A gap theorem of self-shrinkers},
Trans. Amer. Math. Soc. \textbf{367} (2015), no.~7, 4895--4915.

\bibitem{ChengWeiYamashiro2021}
Q.-M.~Cheng, G.~X.~Wei, and T.~Yamashiro,
\emph{Chern conjecture on minimal hypersurfaces},
preprint, arXiv:2104.14057, 2021.

\bibitem{ChengWeiYano2023}
Q.-M.~Cheng, G.~X.~Wei, and W.~Yano,
\emph{The second gap on complete self-shrinkers},
Proc. Amer. Math. Soc. \textbf{151} (2023), no.~1, 339--348; arXiv:2104.14059.

\bibitem{ChengZhou2013}
X.~Cheng and D.~T.~Zhou,
\emph{Volume estimate about shrinkers},
Proc. Amer. Math. Soc. \textbf{141} (2013), no.~2, 687--696.

\bibitem{ChoeSoret2009}
J.~Choe and M.~Soret,
\emph{First eigenvalue of symmetric minimal surfaces in \(\Sphere^3\)},
Indiana Univ. Math. J. \textbf{58} (2009), no.~1, 269--281.

\bibitem{ChoiWang1983}
H.~I.~Choi and A.~N.~Wang,
\emph{A first eigenvalue estimate for minimal hypersurfaces},
J. Differential Geom. \textbf{18} (1983), no.~3, 559--562.

\bibitem{ColdingMinicozzi2012Generic}
T.~H.~Colding and W.~P.~Minicozzi II,
\emph{Generic mean curvature flow I; generic singularities},
Ann.\ of Math. (2) \textbf{175} (2012), no.~2, 755--833.

\bibitem{Ding2018}
Q.~Ding,
\emph{A rigidity theorem on the second fundamental form for self-shrinkers},
Trans. Amer. Math. Soc. \textbf{370} (2018), no.~12, 8311--8329.

\bibitem{DingXin2011}
Q.~Ding and Y.~L.~Xin,
\emph{On Chern's problem for rigidity of minimal hypersurfaces in the spheres},
Adv. Math. \textbf{227} (2011), no.~1, 131--145.

\bibitem{DingXin2013}
Q.~Ding and Y.~L.~Xin,
\emph{Volume growth, eigenvalue and compactness for self-shrinkers},
Asian J. Math. \textbf{17} (2013), no.~3, 443--456.

\bibitem{DingXin2014}
Q.~Ding and Y.~L.~Xin,
\emph{The rigidity theorems of self-shrinkers},
Trans. Amer. Math. Soc. \textbf{366} (2014), no.~10, 5067--5085.

\bibitem{DuncanSireSpruck2024}
J.~A.~J.~Duncan, Y.~Sire, and J.~Spruck,
\emph{An improved eigenvalue estimate for embedded minimal hypersurfaces in the sphere},
Int. Math. Res. Not. IMRN \textbf{2024}, no.~18, 12556--12567.

\bibitem{GeLiuLuoYan2026}
J.~Q.~Ge, T.~Liu, K.~Y.~Luo, and W.~J.~Yan,
\emph{Rigidity of closed minimal hypersurfaces in {$\mathbb S^5$}},
preprint, arXiv:2606.29246, 2026.

\bibitem{GeTanYanZhang2025}
J.~Q.~Ge, H.~X.~Tan, W.~J.~Yan, and Y.~H.~Zhang,
\emph{Chern conjecture on minimal Willmore hypersurfaces with constant scalar curvature},
preprint, arXiv:2512.08342, 2025.

\bibitem{HeXuZhao2026}
C.~He, H.~W.~Xu, and E.~Zhao,
\emph{Classification of closed minimal hypersurfaces with constant scalar curvature in {$\mathbb S^5$}},
preprint, arXiv:2603.01181, 2026.

\bibitem{Huisken1990}
G.~Huisken,
\emph{Asymptotic behavior for singularities of the mean curvature flow},
J. Differential Geom. \textbf{31} (1990), no.~1, 285--299.

\bibitem{Huisken1993}
G.~Huisken,
\emph{Local and global behaviour of hypersurfaces moving by mean curvature},
in \emph{Differential geometry: partial differential equations on manifolds},
Proc. Sympos. Pure Math., vol.~54, Amer. Math. Soc., Providence, RI, 1993,
pp.~175--191.

\bibitem{JimenezChinchayZhou2026}
A.~Jim\'enez, C.~Tapia Chinchay, and D.~T.~Zhou,
\emph{A lower bound for the first eigenvalue of a minimal hypersurface in the sphere},
Rev. Mat. Iberoam. \textbf{42} (2026), no.~1, 261--278.

\bibitem{LeSesum2011}
N.~Q.~Le and N.~\v{S}e\v{s}um,
\emph{Blow-up rate of the mean curvature during the mean curvature flow and a gap theorem for self-shrinkers},
Comm. Anal. Geom. \textbf{19} (2011), no.~4, 633--659.

\bibitem{LeiXuXu2017JFA}
L.~Lei, H.~W.~Xu, and Z.~Y.~Xu,
\emph{On Chern's conjecture for minimal hypersurfaces and rigidity of self-shrinkers},
J. Funct. Anal. \textbf{273} (2017), no.~11, 3406--3425.

\bibitem{LeiXuXu2020}
L.~Lei, H.~W.~Xu, and Z.~Y.~Xu,
\emph{A new pinching theorem for complete self-shrinkers and its generalization},
Sci. China Math. \textbf{63} (2020), no.~6, 1139--1152.

\bibitem{LiWang2017}
H.~Z.~Li and X.~F.~Wang,
\emph{New characterizations of the Clifford torus as a Lagrangian self-shrinker},
J. Geom. Anal. \textbf{27} (2017), no.~2, 1393--1412.

\bibitem{LiZhao2026}
F.~G.~Li and Y.~H.~Zhao,
\emph{Gaussian-weighted curvature gaps for self-shrinkers},
preprint, arXiv:2606.14360, 2026.


\bibitem{Matinpour2025}
E.~Matinpour,
\emph{First eigenvalue and nodal domains of the drift Laplacian on symmetric self-shrinkers in {$\mathbb R^3$}},
arXiv:2509.25617, 2025.

\bibitem{Mantegazza2011}
C.~Mantegazza,
\emph{Lecture notes on mean curvature flow},
Progress in Mathematics, vol.~290, Birkh\"auser/Springer Basel AG, Basel, 2011.

\bibitem{PengTerng1983MathAnn}
C.-K.~Peng and C.-L.~Terng,
\emph{The scalar curvature of minimal hypersurfaces in spheres},
Math. Ann. \textbf{266} (1983), no.~1, 105--113.


\bibitem{SchoenYau1994}
R.~Schoen and S.-T.~Yau,
\emph{Lectures on differential geometry},
Conference Proceedings and Lecture Notes in Geometry and Topology, vol.~1,
International Press, Cambridge, MA, 1994.

\bibitem{Simons1968}
J.~Simons,
\emph{Minimal varieties in Riemannian manifolds},
Ann. of Math. (2) \textbf{88} (1968), no.~1, 62--105.

\bibitem{TangXieYan2014}
Z.~Z.~Tang, Y.~Q.~Xie, and W.~J.~Yan,
\emph{Isoparametric foliation and Yau conjecture on the first eigenvalue, II},
J. Funct. Anal. \textbf{266} (2014), no.~10, 6174--6199.

\bibitem{TangWeiYan2020}
Z.~Z.~Tang, D.~Y.~Wei, and W.~J.~Yan,
\emph{A sufficient condition for a hypersurface to be isoparametric},
Tohoku Math. J. (2) \textbf{72} (2020), no.~4, 493--505.

\bibitem{TangYan2013}
Z.~Z.~Tang and W.~J.~Yan,
\emph{Isoparametric foliation and Yau conjecture on the first eigenvalue},
J. Differential Geom. \textbf{94} (2013), no.~3, 521--540.

\bibitem{TangYan2023}
Z.~Z.~Tang and W.~J.~Yan,
\emph{On the Chern conjecture for isoparametric hypersurfaces},
Sci. China Math. \textbf{66} (2023), no.~1, 143--162.

\bibitem{YangCheng1998}
H.~C.~Yang and Q.-M.~Cheng,
\emph{Chern's conjecture on minimal hypersurfaces},
Math. Z. \textbf{227} (1998), no.~3, 377--390.

\bibitem{Yau1982}
S.-T.~Yau,
\emph{Problem section},
in \emph{Seminar on Differential Geometry},
Ann.\ of Math. Stud., vol.~102, Princeton Univ. Press, Princeton, NJ, 1982,
pp.~669--706.

\bibitem{Zhao2025}
Y.~H.~Zhao,
\emph{A gap theorem on closed self-shrinkers of mean curvature flow},
preprint, arXiv:2503.00505, 2025.

\bibitem{Zhao2026FirstEigenvalue}
Y.~H.~Zhao,
\emph{The first eigenvalue of embedded minimal hypersurfaces in the unit sphere},
preprint, arXiv:2603.20890, 2026.

\end{thebibliography}
\end{document}